\documentclass[1 [leqno,11pt]{amsart}
\usepackage{amssymb, amsmath}
\usepackage{hyperref}
\usepackage{xcolor}
\allowdisplaybreaks[4]

\def\Z{\mathop{\mathbb Z\kern 0pt}\nolimits}
\def\N{\mathop{\mathbb N\kern 0pt}\nolimits}
\def\Q{\mathop{\mathbb Q\kern 0pt}\nolimits}
\def\R{\mathop{\mathbb R\kern 0pt}\nolimits}

\def\l{{\rm l}}

\def\l3{L^3(\R^3)}

\def\ep{\varepsilon}

\newcommand{\Rmnum}[1]{\uppercase\expandafter{\romannumeral #1} }

\newcommand{\beq}{\begin{equation}}
\newcommand{\eeq}{\end{equation}}
\newcommand{\ben}{\begin{eqnarray}}
\newcommand{\een}{\end{eqnarray}}
\newcommand{\beno}{\begin{eqnarray*}}
\newcommand{\eeno}{\end{eqnarray*}}

 \numberwithin{equation}{section}

\newtheorem{thm}{Theorem}[section]
\newtheorem{lem}{Lemma}[section]
\newtheorem{rmk}{Remark}[section]

\newtheorem{prop}{Proposition}[section]

\begin{document}

\title{Swirl ill-posedness for incompressible fluids in (sub)critical spaces}

\author[T. Barker]{Tobias Barker}
\address[T. Barker]
 {University of Bath} 
 \email{tb2130@bath.ac.uk}
\author[R. Hu]{Ruilin Hu}
\address[R. Hu]
 {University of Bath} \email{rh2488@bath.ac.uk}

\date{\today}

\begin{abstract}
For the 3D Navier-Stokes equations with sufficiently smooth initial data having zero swirl $u_{0\theta}\equiv 0$, Ladyzhenskaya \cite{L1968} and Ukhovskii and Yudovich \cite{UY1968} established global regularity. It has since remained open as to whether well-posedness still holds for smooth axisymmetric initial data, with the  initial swirl being smaller than an \textit{absolute} constant in a critical space.

In this remark, we prove \textit{norm inflation} occurs for the Navier-Stokes scalar swirl component, with initial smallness in the critical space $\dot{H}^{\frac{1}{2}}$ or even subcritical Sobolev spaces. We also obtain norm inflation for the 3D Euler swirl \textit{vector} in the Euler subcritical space $\dot{H}^3$. Our results are achieved by introducing an anisotropic scale in Luo's supercritical norm inflation construction \cite{Luo24}, as well as performing an anisotropic bootstrap argument to get a sufficiently small swirl remainder term. 
\end{abstract}
\maketitle

\vskip-0.3cm

\setcounter{equation}{0}

	\section{Introduction}
The three-dimensional incompressible Navier–Stokes and Euler equations play a fundamental role in mathematical fluid mechanics and take the form
\begin{equation}\label{eq1.1}
    \left\{\begin{array}{l}
         \partial_tu-\nu\Delta u+u\cdot\nabla u=\nabla P,\quad\forall (x,t)\in\mathbb{R}^3\times (0,T) , \\
          \operatorname{div}u=0,\quad\forall (x,t)\in\mathbb{R}^3\times (0,T),\\
          u|_{t=0}=u_0\quad\forall x\in\mathbb{R}^3.
    \end{array} \right.
\end{equation}
 Throughout this paper we will take either $\nu=1$ (Navier-Stokes equations) or $\nu=0$ (Euler equation). 

Global regularity of the unforced 3D Navier-Stokes equation remains open. A natural avenue has been to investigate regularity of these equations in the presence of certain symmetries. One example is that of axisymmetry, which is preserved by the evolution of the equation. For sufficiently smooth axisymmetric initial data with no swirl, global regularity was established by Ladyzhenskaya \cite{L1968} and independently by Ukhovskii and Yudovich \cite{UY1968}.  For axisymmetric solutions, it is natural then to ask if global regularity occurs for initial data with small enough swirl. For initial data with swirl that is small enough with respect to other components `relatively small swirl', this was accomplished in \cite{CFZ2017,L2022,LZ2017,W2016,ZZ2014}  for instance. For the 3D Euler equations, initial data with small initial swirl compared to the other axisymmetric components was shown to extend the life span of the solution in \cite{D2013}, \cite{LZ2024} and \cite{LXY26}.

From a qualitative perspective, for the Navier-Stokes equations the case of relatively small initial swirl can be reduced to the case of no swirl. In particular for all $M>0$ there exists $\varepsilon(M)>0$ such that for $H^2$ axisymmetric initial data\footnote{In \eqref{relativelysmall} and the argument that follows, the $L^2$ norm could be replaced with any norm $X$ with $\varepsilon$ depending on $M$ and $\|\cdot\|_{X}$. }  with
\begin{equation}\label{relativelysmall}
\|u_0\|_{H^{2}(\mathbb{R}^3)}\leq M\quad\textrm{and}\quad \|u_0\cdot e_{\theta}\|_{L^{2}(\mathbb{R}^3)}\leq \varepsilon(M),
\end{equation}
the associated suitable\footnote{This means a weak Leray-Hopf solution (for which definition can be seen in Definition 5.1 p.102 in \cite{S2014}) that satisfies the local energy inequality given in Definition 6.1 p.133 in \cite{S2014}.} weak Leray-Hopf solution is smooth on $\mathbb{R}^3\times (0,\infty)$. To achieve this, the same compactness arguments can be used as in \cite[First step- second step p.66]{BP2021} and \cite{B2025}. These arguments show that the non-existence of such an $\varepsilon(M)$ in \eqref{relativelysmall} yields a singular axisymmetric suitable weak Leray-Hopf solution with $H^{2}$ axisymmetric initial data without swirl, which contradicts \cite{L1968}, \cite{UY1968} and \cite{LMNP1999}.

 For axisymmetric initial data that has swirl smaller than an absolute value in critical space\footnote{A Banach $X$ is called critical space if the corresponding norm $\|\cdot\|_X$ is invariant with the Navier-Stokes scaling of initial data $u_{0\lambda}(x)=\lambda u_0(\lambda x)$.}, the situation is more delicate and unclear\footnote{See for instance a related question stated in Remark 1.1 of \cite{L2022}.}. In particular, the following problem is open:
 \begin{itemize}
 \item[]
\textbf{(Q)} For the axisymmetric Navier-Stokes equations in $\mathbb{R}^3$, does there exist a universal constant $\ep>0$ and a critical space $X$ such that if the initial data $u_0$ satisfies $$\|u_0\cdot e_\theta\|_{X}<\ep,$$ then the solution is regular globally in time?
\end{itemize}

In this paper, we show that for axisymmetric initial data instantaneous norm inflation occurs for the 3D Navier-Stokes equations on the level of the swirl in the critical space $\dot{H}^{\frac{1}{2}}$. This partially answers a question raised in Remark 1.1 (i) (end of paragraph) in \cite{L2022}. We also show that such instantaneous norm inflation even occurs in certain subcritical\footnote{A Banach $Y$ is called  a subcritical space  if under the Navier-Stokes  initial data scaling $u_{0\lambda}(x):=\lambda u_0(\lambda x)$ ($\lambda>0$), we have $\|u_{0\lambda}\|_{Y}=\lambda^{\alpha}\|u_0\|_{Y}$ for some $\alpha>0$.} spaces for the swirl of the 3D Navier-Stokes equations. For the axisymmetric 3D Euler equations,  instantaneous norm inflation is demonstrated for the vector swirl in $\dot{H}^3$. For divergence-free initial data, it is known that all components belonging to $\dot{H}^3$ ensures local well-posedness, with the solution bounded in terms of $\|u_0\|_{\dot{H}^3}$ on the time interval $[0, c\|u_0\|_{\dot{H}^3}^{-1}]$.

These results demonstrate that instability of the swirl component of axisymmetric solutions persists even in high regularity spaces,  due to strong nonlinear coupling. These results are stated precisely in the next section.

\subsection{Main results}

\subsubsection{Swirl inflation for axistmmetric solutions}\label{sec1.1.3}

In what follows `sufficiently small' means smaller than a certain absolute/universal constant.

For $x=(x_1,x_2,x_3)\in\mathbb{R}^3$, we define
\begin{equation}\label{eq1.19}
    x_1=r\cos\theta,\quad x_2=r\sin\theta, \quad x_3=z,
\end{equation}
and the orthogonal unit vectors
\begin{equation}\label{eq1.20}
    e_r=(\cos\theta,\sin\theta,0),\quad e_\theta=(-\sin\theta,\cos\theta,0),\quad e_z=(0,0,1)
\end{equation}
So for any vector field $v=(v_1,v_2,v_3)$, we can rewrite as
\begin{equation*}
    v=v_re_r+v_\theta e_\theta+v_ze_z,
\end{equation*}
with $v_r=v\cdot e_r$, $v_\theta=v\cdot e_\theta,$ and $v_z=v\cdot e_z$. We call a function $f$ {axisymmetric} if it is independent of $\theta$, $\partial_\theta f=0$. Similarly we call a vector field $v$ axisymmetric if all components $v_r,v_\theta,v_z$ are axisymmetric.

\begin{thm}[Navier-Stokes (sub)critical swirl inflation]\label{thm7}
    Fix $\eta\in[0,\frac{1}{2}]$. For sufficiently small $\ep>0$, there exists $T_*\in (0,\varepsilon]$ and an axisymmetric initial data $u_0=u_{0r}e_{r}+u_{0z}e_{z}+u_{0\theta}e_{\theta}$ with $u_0\in C^{\infty}_{0}(\mathbb{R}^3)$ such that the following holds.
  \begin{enumerate}  
   \item  The scalar swirl component of the initial data $u_0$ satisfies
    \begin{equation*}\label{eq1.13}
        \|u_{0\theta}\|_{\dot H^{\frac{1}{2}+\eta}}\leq\ep.
    \end{equation*}

 \item There is a unique axisymmetric solution $u=u_{r}e_{r}+u_{z}e_{z}+u_{\theta}e_{\theta}$ in $C([0,T_*];H^1)$ to the Navier-Stokes equations \eqref{eq1.1} where $\nu=1$, with initial data $u_0$.  The scalar swirl component of the solution $u$ has lower bound
    \begin{equation*}\label{eq1.14}
        \|u_{\theta}(\cdot, T_*)\|_{\dot H^{\frac{1}{2}+\eta}}\geq \ep^{-1}.
    \end{equation*}
    \end{enumerate}
\end{thm}
\begin{thm}[Euler subcritical swirl inflation]\label{thm6}
 For sufficiently small $\ep>0$, there exists $T_*\in (0,\varepsilon]$ and an axisymmetric initial data $u_0=u_{0r}e_{r}+u_{0z}e_{z}+u_{0\theta}e_{\theta}$ with $u_0\in C^{\infty}_{0}(\mathbb{R}^3)$ such that the following holds.
  \begin{enumerate}  
   \item  The vector swirl of the initial data $u_0$ satisfies
    \begin{equation*}\label{eq1.11}
        \|u_{0\theta} e_\theta\|_{\dot H^{3}}\leq\ep.
    \end{equation*}

 \item There is a unique axisymmetric solution $u=u_{r}e_{r}+u_{z}e_{z}+u_{\theta}e_{\theta}$ in $C([0,T_*];H^3)$ to the Euler  equations \eqref{eq1.1} where $\nu=0$, with initial data $u_0$.  The vector swirl  of the solution $u$ has lower bound
    \begin{equation*}\label{eqEulerinflation}
         \|u_{\theta}(\cdot, T_*)e_\theta\|_{\dot H^{3}}\geq \ep^{-1}.
    \end{equation*}
    \end{enumerate}
\end{thm}

\subsection{Remarks}

\subsubsection{Role of critical norms and small swirl regularity criteria}
Let us mention that there are critical norms with respect to the Navier-Stokes scaling $\lambda f(\lambda x)$ (that are not translation invariant) such that the swirl does not inflate. In particular, for a sufficiently smooth axisymmetric solution with swirl $u_{\theta}:=u\cdot e_{\theta}$, the apriori estimate holds
\begin{equation}\label{swirlapriori}
\|ru_{\theta}(r,z,t)\|_{L^{\infty}(\mathbb{R}^2)}\leq \|ru_{\theta}(r,z,0)\|_{L^{\infty}(\mathbb{R}^2)},
\end{equation}
so norm inflation cannot occur in such a critical space.
For special domains that are exterior to particular cones and with certain boundary conditions imposed, it is shown in \cite{LZP2024} and \cite{LYZ2026} that initial data with small swirl in this critical space produces a global unique strong solution of the Navier-Stokes equations.

Let us also mention that it is not known whether the quantity in \eqref{swirlapriori} being small with respect to some absolute constant on some time interval suppresses singularities. At the time of writing, some decay in $r$ is required as $r\rightarrow 0$. See \cite{CFZ2017}, \cite{LQZ2017}, \cite{W2016} and \cite{palasek26} for example.  We cannot pretend to give an exhaustive account of regularity criteria in the axisymmetric setting and recommend the excellent review article \cite{ZP2022}.

\subsubsection{Role of vector swirl for the Euler equations}
The swirl inflation result for the 3D Euler equations is stated for the vector swirl rather than the scalar swirl.
This is because the local well-posedness Sobolev spaces for 3D Euler are $\dot{H}^{s}$ with $s>\tfrac{5}{2}$, yet a smooth  divergence-free axisymmetric vector field may have an associated scalar swirl that fails to be $\dot{H}^{2}$. This can be seen for the smooth divergence-free axisymmetric vector
$f(r^2+z^2)re_{\theta}$, where $f$ is a smooth compactly supported function equal to $1$ on $\{x:|x|\leq \frac{1}{2}\}$. Let us mention that Theorem \ref{thm6} also holds true for the $\dot{H}^{3}$ norm with $\dot{H}^{k}$ for any fixed $k\in\mathbb{N}$. The proof is virtually unchanged, up to a choice of parameters.
\subsection{Approach}
We will give some heuristics for proving the theorems above in this section, which rely upon Luo's seminal work  \cite{Luo24}\footnote{For further developments on supercritical norm inflation for Navier-Stokes equation, see \cite{Luo2025}.}. To fix ideas we focus on the Navier-Stokes case.

To achieve Theorem \ref{thm7} we utilise Luo's axisymmetric pseudo solution framework in \cite{Luo24}. The radial and vertical components of our initial data  are the same as in Luo's proof in \cite{Luo24} of $\dot{H}^{s}$ norm inflation ($s\in(0,\frac{1}{2})$) for the 3D Navier-Stokes equations, as in \cite{Luo24} this quantitatively supercritical background is what drives the inflation. In order to have a small swirl component in the (sub)critical space $\dot{H}^{\frac{1}{2}+\eta}$, we choose the swirl component to have an anisotropic scale, different to Luo's initial data. As in \cite{Luo24}, the approximate solution's angular component solves a transport equation that experiences growth in Sobolev norms, which drives the inflation. However, in our case the growth is quantitatively smaller than in \cite{Luo24} and the remainder estimate between the actual solution and the approximate solution in \cite{Luo24} is too large to infer inflation for the swirl component of the actual solution. We resolve this by anisotropically bootstrapping the additional smallness of the approximate solution's swirl into a more precise estimate for the swirl of the remainder. This allows us to infer swirl inflation for the actual solution.

\subsection{Notation and preliminaries}
We will list some of the notations we will frequently use in the following text. 
\begin{itemize}
    \item We denote $A\lesssim B$ if there exists some universal constant $C>0$, such that $A\leq CB$. Similarly, we can define $A\gtrsim B$. We denote $A\sim B$ if $A\lesssim B$ and $B\lesssim A$. Furthermore, for parameters $\{\mu\}_{i=1}^k$, we use $A\lesssim_{\mu_1,...,\mu_k}B$ to denote $A\leq C(\mu_1,...,\mu_k)B$ for some constant $C(\mu_1,...,\mu_k)$ which only depends on the parameters $\{\mu_i\}_{i=1}^k$.
    \item For any $\alpha,\beta\in\mathbb{N}^d$, we say $\alpha\leq \beta$ if for any $1\leq i\leq d$, there holds $\alpha_i\leq \beta_i$. If there exists at least one $j$ such that $\alpha_j<\beta_j$, then we say $\alpha<\beta$.
    
    \item If $X$ is a Banach space with norm $\|\cdot\|_{X}$, then $L^{s}(a,b;X)$, with $a<b$ and $s\in[1,\infty)$,  will denote the usual Banach space of strongly measurable $X$-valued functions $f(t)$ on $(a,b)$ such that
$$\|f\|_{L^{s}(a,b;X)}:=\left(\int\limits_{a}^{b}\|f(t)\|_{X}^{s}dt\right)^{\frac{1}{s}}<+\infty.$$ 
The usual modification is made if $s=\infty$. Let $C([a,b]; X)$ denote the space of continuous $X$ valued functions on $[a,b]$ with usual norm.
    \item For any $\alpha \in \mathbb{N}^d$, we define the operator $\nabla^\alpha$ by
    \begin{equation*}
        \nabla^\alpha f=\partial_{x_1}^{\alpha_1}\cdots\partial_{x_d}^{\alpha_d}f,
\qquad \alpha=(\alpha_1,\cdots,\alpha_d).
    \end{equation*}
    for any scalar or vector function $f$. Furthermore, for any $k\in\mathbb{N}$, we define
    \begin{equation*}
        \|\nabla^kf\|_{X}:=\sum_{|\alpha|=k}\|\nabla^\alpha f\|_X.
    \end{equation*}
    \item For any $k\in\mathbb{N}$, $p\in[1,\infty]$, we define the homogeneous Sobolev space
    \begin{equation*}
       \dot  W^{k,p}:=\left\{ f:\|f\|_{\dot W^{k,p}}:=\|\nabla^kf\|_{L^p}<\infty \right\},
    \end{equation*}
    and for the corresponding inhomogeneous space
    \begin{equation*}
       W^{k,p}:=\left\{ f:\|f\|_{ W^{k,p}}:=\sum_{j\leq k}\|\nabla^jf\|_{L^p}<\infty \right\}.
    \end{equation*}
    Explicitly, for the case $p=2$ and any $s\in\mathbb{R}_+$, we define the space $\dot H^{s}$ and $H^s$ by the following norm
    \begin{equation*}
        \|f\|^2_{\dot H^s}=\int_{\mathbb{R}^d}|\xi|^{2s}|\hat{f}(\xi)|^2d\xi,\quad \|f\|_{H^s}^2=\int_{\mathbb{R}^d}(1+|\xi|^2)^{s}|\hat{f}(\xi)|^2d\xi.
    \end{equation*}
    The following interpolation inequality holds in Sobolev spaces
    \begin{equation}\label{eq4.122}
        \|f\|_{\dot H^{\alpha}}\leq\|f\|_{\dot H^{\alpha_0}}^{\theta}\|f\|_{\dot H^{\alpha_1}}^{1-\theta},\quad \alpha=\theta \alpha_0+(1-\theta)\alpha_1,
    \end{equation}
    for any $\alpha_0\leq \alpha\leq \alpha_1$. The proof of this inequality above can be seen in Proposition 1.32 of \cite{BCD2011}. 

    \item For a smooth axisymmetric vector field $f=f_{r}(r,z)e_{r}+f_{\theta}(r,z)e_{\theta}+f_{z}(r,z)e_{z}$,
    \begin{equation}\label{matrixnorm}
    |\nabla^{k}(f_{\star}(r,z)e_{\star})|\leq |\nabla^k f|\quad\textrm{for}\quad \star\in\{r,\theta, z\}.
    \end{equation}
    This follows directly from the fact that
   $$
    \nabla^\alpha \left(f_\theta(r,z)\frac{y}{r}\right)\nabla^\alpha \left(f_r(r,z)\frac{x}{r}\right)=0\quad\textrm{on}\quad\{(x,y,z):y=0,\,x\neq 0\},
$$
due to  the function  being odd in $y$ and  the smoothness of $f_\theta(r,z)\frac{y}{r}$ and $f_r(r,z)\frac{x}{r}$ by Corollary 1 of \cite{LW2009}.
    \end{itemize}

\section{Navier-Stokes case}\label{sec4.1}
\subsection{Construction of initial data}\label{sec4.1.1}
As in \cite{Luo24}, we will also study the problem on the shifted $rz-$plane, where the auxiliary parameters $(\rho,\varphi)$ are defined as
\begin{equation}\label{eq4.2}
    r=\nu^{-1}+\rho \cos\varphi,\quad z=\rho\sin\varphi,\quad \rho\geq 0,\quad\varphi\in[0,2\pi].
\end{equation}
Here $\nu$ is a large constant which will be determined later, see also \eqref{eq4.1}. 
    Fix $\eta\in[0,\frac{1}{2}]$ and $s\in(0,\frac{1}{2})$. As in \cite[(5.1)]{Luo24} let $\mu\gg 1$, and set $b=\frac{\frac{1}{2}-s}{100}>0$. We define
\begin{equation}\label{eq4.1}
    \nu=\mu^{1-b}\ll\mu,\quad T_*=\ep^{-2-N}\mu^{s-2}\nu^{-\frac{1}{2}},\quad N=\max\{\frac{10}{s},100\},
\end{equation}
for some $\ep$ to be fixed later. These are as in \cite[(3.6), (3.22)-(3.23)]{Luo24}. Below, with a slight abuse of notation, we will not distinguish $f(r,z)$ and $f(\rho,\varphi)$.

We define our initial data as
\begin{equation}\label{eq4.4}
    \begin{aligned}
    &u_0=u_{0r}e_r+u_{0\theta}e_\theta+u_{0z}e_z+u_ce_z,\\
    &u_{0\theta}=\ep^2\mu^{\frac{1}{2}-\eta}\nu^{\frac{1}{2}}g(\mu\rho)\sin\varphi,\\
        &u_{0r}=-\ep^2\mu^{1-s}\nu^{\frac{1}{2}}f'(\mu\rho)\partial_z\rho,\\
        &u_{0z}=\ep^2\mu^{1-s}\nu^{\frac{1}{2}}f'(\mu\rho)\partial_r\rho,\\
        &u_c=\ep^2\mu^{-s}\nu^{\frac{1}{2}}\frac{f(\mu\rho)}{r},
    \end{aligned}
\end{equation}
for some $\ep\ll 1$, with $\eta\in[0,\frac{1}{2}]$ and $s\in(0,\frac{1}{2})$ fixed as above, and two non-zero $C_c^\infty$ functions $f,g$ satisfying 
\begin{equation}\label{eq4.100}
    \begin{aligned}
        &\operatorname{supp} g\subset \{\rho:|\rho|\in [1,\tfrac{3}{2}]\},\quad f'(x)=1\quad\textrm{on}\quad\{\rho:|\rho|\in [1,\tfrac{3}{2}]\} ,\\
        &\operatorname{supp} f\subset \{x:|x|\in [\tfrac{1}{2},2]\}.
\end{aligned}
\end{equation}
Note that the term $u_{c}$ in \eqref{eq4.4} is the divergence corrector, to ensure $\nabla\cdot u_0=0$. Observe that $u_{0r}$, $u_{0z}$ and $u_c$ are as in \cite{Luo24}, whereas $u_{0\theta}$ differs from \cite{Luo24} by a multiplicative scale $\mu^{s-\frac{1}{2}-\eta}$.

The initial data then has the following properties.
\begin{prop}\label{prop4.1}
    Let $u_0$ be the initial data defined in \eqref{eq4.4}, then $u_0$ satisfies the following:\\
    i) $u_0$ is smooth and divergence-free.\\
    ii) $u_0$ has compact support in $\{(r,z):|r-\nu^{-1}|\leq 2\mu^{-1},|z|\leq 2\mu^{-1}\}$.\\
    iii) For any $k\in\mathbb{N}$ and $p\in[1,\infty]$, there holds
    \begin{equation}\label{eq4.6}
        \begin{aligned}
            &\|u_{0r}\|_{\dot W^{k,p}}+\|u_{0z}\|_{\dot W^{k,p}}\lesssim_{k,p}\ep^2\mu^{k-s}(\mu^{1-\frac{2}{p}}\nu^{\frac{1}{2}-\frac{1}{p}}),\\
            &\|u_{0\theta}\|_{\dot W^{k,p}}\lesssim_{k,p}\ep^2\mu^{k-\frac{1}{2}-\eta}(\mu^{1-\frac{2}{p}}\nu^{\frac{1}{2}-\frac{1}{p}}),\\
            &\|u_{c}\|_{\dot W^{k,p}}\lesssim_{k,p}\ep^2\mu^{k-1-s}\nu(\mu^{1-\frac{2}{p}}\nu^{\frac{1}{2}-\frac{1}{p}}).
        \end{aligned}
    \end{equation}
    In particular,
    \begin{equation}\label{eq4.7}
        \|u_0\|_{\dot H^{s}}\leq C_1\ep^2,\quad \|u_{0\theta}\|_{\dot H^{\frac{1}{2}+\eta}}\leq C_1\ep^2,
    \end{equation}
    for some universal constant $C_1$.
\end{prop}
The proof is essentially contained in \cite[(3.15) and Lemma 3.2]{Luo24}, where the swirl data is
$$\ep^2\mu^{1-s}\nu^{\frac{1}{2}}g(\mu\rho)\sin\varphi.$$
In \eqref{eq4.4}, our swirl initial data differs from \cite{Luo24} by a multiplicative scale $\mu^{s-\frac{1}{2}-\eta}\ll 1$. So from \cite[(3.16)]{Luo24}, there holds
\begin{equation*}
    \|u_{0\theta}\|_{\dot W^{k,p}}\lesssim_{k,p} \mu^{s-\frac{1}{2}-\eta}\ep^2\mu^{k-s}(\mu^{1-\frac{2}{p}}\nu^{\frac{1}{2}-\frac{1}{p}})=\ep^2\mu^{k-\frac{1}{2}-\eta}(\mu^{1-\frac{2}{p}}\nu^{\frac{1}{2}-\frac{1}{p}}).
\end{equation*}
Then, \eqref{eq4.7} follows from the estimate above and the Sobolev interpolation inequality \eqref{eq4.122}.
\subsection{Analysis of approximate system}\label{sec4.1.2}

Let us recall the approximate solution framework from \cite{Luo24}, and derive the corresponding estimates for our setting. As mentioned in \cite{Luo24}, $(u_{0r},u_{0z})$ is a solution to the stationary 2D Euler equations \begin{equation}\label{eq4.9}
  \left\{\begin{array}{l}
  (u_{0r}\partial_r+u_{0z}\partial_z)u_{0r}+\partial_rP=0, \\
  (u_{0r}\partial_r+u_{0z}\partial_z)u_{0z}+\partial_zP=0, \\
    \partial_ru_{0r}+\partial_zu_{0z}=0,
    \end{array}\right.
\end{equation}
 As in \cite{Luo24}, define the approximate solution \begin{equation}\label{eq4.10}
    \bar u=\bar u_re_r+\bar u_\theta e_\theta+\bar u_ze_z.
\end{equation} 
Here $\bar u_r= u_{0r}$, $\bar u_z=u_{0z}+u_c$  and $\bar u_\theta$ satisfies the transport equation \begin{equation}\label{eq4.11}
   \left\{\begin{array}{l}
   \partial_t\bar u_\theta+\left(u_{0r}\partial_r+u_{0z}\partial_z\right)\bar u_\theta=0,\\
 \bar u_\theta|_{t=0}=u_{0\theta},
 \end{array}\right.
\end{equation}
 which has the explicit formula
\begin{equation}\label{eq4.12}
    \bar u_\theta=\ep^2\mu^{\frac{1}{2}-\eta}\nu^{\frac{1}{2}}g(\mu\rho)\sin(\varphi-\ep^2t\rho^{-1}\mu^{1-s}\nu^{\frac{1}{2}}).
\end{equation}
As discussed in \cite{Luo24}, the approximate solution $\bar u$ defined in \eqref{eq4.10} is axisymmetric, smooth and divergence-free, and satisfies the following system
\begin{equation}\label{eq4.13}
    \begin{aligned}
        \left\{\begin{array}{l}
              \partial_t\bar u_r-(\partial_r^2+\frac{\partial_r}{r}+\partial_z^2-\frac{1}{r^2})\bar u_r+(\bar u_r\partial_r+\bar u_z\partial_z)\bar u_r-\frac{\bar u_\theta ^2}{r}=E_r+\partial_rP,\\
              \partial_t\bar u_\theta-(\partial_r^2+\frac{\partial_r}{r}+\partial_z^2-\frac{1}{r^2})\bar u_\theta+(\bar u_r\partial_r+\bar u_z\partial_z)\bar u_\theta+\frac{\bar u_\theta \bar u_r}{r}=E_\theta,\\
              \partial_t\bar u_z-(\partial_r^2+\frac{\partial_r}{r}+\partial_z^2)\bar u_z+(\bar u_r\partial_r+\bar u_z\partial_z)\bar u_z=E_z+\partial_zP,
        \end{array}\right.
    \end{aligned}
\end{equation}
with the force term $(E_r,E_\theta,E_z)$ given by
\begin{equation*}\label{eq4.14}
        \left\{\begin{array}{l}
             E_r=u_c\partial_zu_{0r}-\frac{1}{r}\bar u_\theta^2-(\partial_r^2+\frac{\partial_r}{r}+\partial_z^2-\frac{1}{r^2})\bar u_r,  \\
             E_\theta=u_c\partial_z\bar u_\theta+\frac{1}{r}\bar u_\theta u_{0r}-(\partial_r^2+\frac{\partial_r}{r}+\partial_z^2-\frac{1}{r^2})\bar u_\theta, \\
             E_z=u_{0r}\partial_ru_c+u_{0z}\partial_zu_c+u_c\partial_zu_c+u_c\partial_zu_{0z}-(\partial_r^2+\frac{\partial_r}{r}+\partial_z^2)\bar u_z.
        \end{array}\right.
\end{equation*}
Moreover, $\bar u$ is the solution to the following system
\begin{equation}\label{eq4.103}
    \left\{\begin{array}{l}
         \partial_t\bar u-\Delta \bar u+\bar u\cdot\nabla\bar u=E+\nabla P,  \\
          \operatorname{div} \bar u=0,\\
          \bar u|_{t=0}=\bar u_0,
    \end{array}
    \right.
\end{equation}
where $E=E_re_r+E_\theta e_\theta+E_ze_z$. It is worth mentioning that both the approximate solution $\bar u$ and the perturbation force $E$ are supported on a ring with distance $\nu^{-1}$ from the $z-$axis. 

The following proposition gives estimates for $\bar u$ for our setting.
\begin{prop}\label{prop4.8}
    The function $\bar u$ is a smooth axisymmetric solution to \eqref{eq4.13} on $[0,T_*]$, with the following estimate for any $k\in\mathbb N$ and $p\in[1,\infty]$,
    \begin{equation}\label{eq4.15}
        \begin{aligned}
            &\|\bar{u}\|_{L^{\infty}(0,T_*;\dot{W}^{k,p})}+\|\bar u_re_r+\bar u_ze_z\|_{L^\infty(0,T_*;\dot W^{k,p})}\lesssim_{k,p} \ep^{2-Nk} \mu^{k-s}(\mu^{1-\frac{2}{p}}\nu^{\frac{1}{2}-\frac{1}{p}}),\\
            &\|\bar u_\theta\|_{L^\infty(0,T_*;\dot W^{k,p})}+\|\bar u_\theta e_\theta\|_{L^\infty(0,T_*;\dot W^{k,p})}\lesssim_{k,p} \ep^{2-Nk} \mu^{k-\frac{1}{2}-\eta}(\mu^{1-\frac{2}{p}}\nu^{\frac{1}{2}-\frac{1}{p}}).
        \end{aligned}
    \end{equation}
    Moreover, the remainder force term $E$ satisfies the following estimate
    \begin{equation}\label{eq4.16}
        \begin{aligned}
            &\|E\|_{L^\infty(0,T_*;\dot H^{k})}\lesssim_{\ep,k}   \mu^{k-s}(\mu^{2-s}\nu^{\frac{1}{2}})(\mu^{-1}\nu),\\
        &\|E_\theta\|_{L^\infty(0,T_*;\dot H^{k})}+\|E_\theta e_\theta\|_{L^\infty(0,T_*;\dot H^{k})}\lesssim_{\ep,k} \mu^{k-\frac{1}{2}-\eta}(\mu^{2-s}\nu^{\frac{1}{2}})(\mu^{-1}\nu).
        \end{aligned}
    \end{equation}
\end{prop}
\begin{proof}
In \cite{Luo24}, the approximate solution's  swirl component is
$$\ep^2\mu^{1-s}\nu^{\frac{1}{2}}g(\mu\rho)\sin(\varphi-\ep^2t\rho^{-1}\mu^{1-s}\nu^{\frac{1}{2}}),$$ so our approximate solution's swirl $\bar{u}_{\theta}$ in \eqref{eq4.12} differs by a coefficient $\mu^{s-\frac{1}{2}-\eta}\ll 1$ for $\mu\gg 1$. Note as well that $\bar{u}_{r}$ and $\bar{u}_{z}$ are the same as in \cite{Luo24}. Using these facts, we inherit the estimates in \cite[Proposition 3.4]{Luo24} for $\bar{u}$ and $E$ (denoted $\bar{E}$ in \cite{Luo24}) given by \eqref{eq4.16} and \eqref{eq4.15}. Moreover, the estimates for $\bar{u}_{\theta}$ and $E_{\theta}$ differ from those in \cite[(3.26) and (3.30)]{Luo24} by a multiplicative constant $\mu^{s-\frac{1}{2}-\eta}$ giving the estimates in \eqref{eq4.15}-\eqref{eq4.16} for $\bar{u}_{\theta}$ and $E_{\theta}$.

For estimating $\bar{u}_{\theta}e_{\theta}$, we use $|\nabla^je_{\theta}|\lesssim_j \nu^{j}$ on the support of $\bar{u}_{\theta}$, the product rule, H\"{o}lder's inequality and the estimate for $\bar{u}_{\theta}$. This gives 
\begin{equation}\label{eq4.128}
    \begin{aligned}
        \|\bar u_\theta e_\theta\|_{L^\infty(0,T_*;\dot W^{k,p})}&\lesssim_{k,p} \sum_{j\leq k}\nu^{k-j}\|\bar u_\theta\|_{L^\infty(0,T_*;\dot W^{j,p})}
        \\&\leq \sum_{j\leq k} \mu^{k-\frac{1}{2}-\eta }\varepsilon^{2-Nj}(\mu^{1-\frac{2}{p}}\nu^{\frac{1}{2}-\frac{1}{p}})\lesssim_{k} \varepsilon^{2-Nk} \mu^{k-\frac{1}{2}-\eta}(\mu^{1-\frac{2}{p}}\nu^{\frac{1}{2}-\frac{1}{p}}),
    \end{aligned}
\end{equation} 
where we used $\nu\leq \mu$ and $\varepsilon\in (0,1)$.
The estimates of $\bar u_re_r+\bar u_ze_z$ and $E_{\theta}e_{\theta}$ are similar.
\end{proof}

In the proposition below, we will show that the $\bar u_\theta$ does have the desired norm inflation at the critical time $T_*$.
\begin{prop}\label{prop4.5}
    For $\bar u_\theta$ defined by \eqref{eq4.12}, $\ep\in(0,\frac{1}{2})$, and $T_*$ defined by \eqref{eq4.1}, there holds
    \begin{equation}\label{eq4.34}
        \|\bar u_\theta(\cdot,T_*)\|_{\dot H^{\frac{1}{2}+\eta}}\gtrsim \ep^{-2}.
    \end{equation}
\end{prop}
\begin{proof}
The proof closely follows  \cite[Lemma 3.5]{Luo24}. As the approximate swirl $\bar{u}_{\theta}$ and inflation norm differ in our setting we provide a sketch. 

As the swirl $\bar{u}_{\theta}$ in \eqref{eq4.12} differs from \cite{Luo24} by a coefficient $\mu^{s-\frac{1}{2}-\eta}$, the same reasoning as in \cite[(3.41)]{Luo24} gives 

\begin{equation*}\label{eq4.101}
    \|\bar u_\theta(\cdot,T_*)\|_{\dot H^1}\gtrsim \ep^{2-N}\mu^{\frac{1}{2}-\eta}.
\end{equation*}
Sobolev interpolation \eqref{eq4.122}, together with the $\dot H^2$ upper bound \eqref{eq4.15}, leads to 
\begin{equation*}\label{eq4.102}
    \|\bar u_\theta(\cdot,T_*)\|_{\dot H^{\frac{1}{2}+\eta}}\geq \|\bar u_\theta(\cdot,T_*)\|_{\dot H^1}^{\frac{3}{2}-\eta}\|\bar u_\theta(\cdot,T_*)\|_{\dot H^2}^{\eta-\frac{1}{2}}\gtrsim \varepsilon^{2-N-N(\eta-\frac{1}{2})}\geq \ep^{-2}.
\end{equation*}
\end{proof}
\begin{rmk}\label{rmkvectorinflation}
As the swirl component \eqref{eq4.12} differs from \cite{Luo24} by the multiplicative coefficient $\mu^{s-\frac{1}{2}-\eta}$, we see that the same reasoning as in \cite[(3.40)]{Luo24} also gives
\begin{equation}\label{vecinflation}
\|\bar{u}_{\theta}(\cdot, T_*)e_{\theta}\|_{\dot{H}^{1}}\geq C \varepsilon^{2-N} \mu^{\frac{1}{2}-\eta},
\end{equation}
for some universal constant $C$. This will be of particular use in the proof of Theorem \ref{thm6}.
\end{rmk}
\subsection{Error estimates}\label{sec4.1.3}

This subsection is devoted to the estimate of the remainder term $w=u-\bar u$, where $u$ is the axisymmetric solution to the Navier-Stokes equation with the initial data $u_0$ defined in \eqref{eq4.4}, and $\bar u$ is the approximate solution \eqref{eq4.10}. Combining \eqref{eq1.1} and \eqref{eq4.103}, $w$ is a solution to the following system
\begin{equation}\label{eq4.23}
    \left\{\begin{array}{l}
         \partial_tw-\Delta w+u\cdot\nabla w+w\cdot\nabla\bar u=\nabla P-E,  \\
         \operatorname{div} w=0,\\
          w|_{t=0}=0.
    \end{array}\right.
\end{equation}
The following proposition gives an estimate for the remainder $w$.
\begin{prop}\label{prop4.3}
    Let $\bar u$ be defined by \eqref{eq4.10}. If the force term $E$ satisfies \eqref{eq4.16}, then there exists a solution $w:\mathbb{R}^3\times(0, T_*]\rightarrow \mathbb{R}^3$ to \eqref{eq4.23} with the following properties. Namely, for each integer $k\in\mathbb{N}$, there exists a sequence $\delta_k=\frac{9+10^{-k}}{10}$ for $k>0$ and $\delta_0=1$, such that
    \begin{equation}\label{eq4.67}
        \begin{aligned}
            \|w(\cdot, t)\|_{\dot{H}^k}\lesssim_{\ep,k}\mu^{k-s}(\mu^{-1}\nu)^{\delta_k},\quad\forall t\leq T_*.
        \end{aligned}
    \end{equation}
    Moreover, there holds the pointwise estimate
    \begin{equation}\label{eq4.125}
        \begin{aligned}
            \|w(\cdot, t)\|_{\dot{W}^{k,\infty}}\lesssim_{\ep,k}\mu^{k+\frac32-s}(\mu^{-1}\nu)^{\delta_{k+3}},\quad\forall t\leq T_*.
        \end{aligned}
    \end{equation}
\end{prop}
\begin{proof}
For a detailed proof of the $L^2$ estimate \eqref{eq4.67}, we refer to  \cite[Lemma 4.2-4.3]{Luo24}. We emphasize that the only difference is that our swirl data $u_{0\theta}$ is smaller, so the whole proof can be inherited from \cite{Luo24}, since only an upper bound of data is used.

Following \cite{Luo24}, one can use \eqref{eq4.67} and the Sobolev embedding theorem to show \eqref{eq4.125}.
In particular, by the interpolation inequality \eqref{eq4.122} and \eqref{eq4.67}, 
    \begin{equation*}\label{eq4.123}
       \|w(\cdot,t)\|_{\dot H^{k+\frac32+\delta}}\lesssim_{\ep,k,\delta}\mu^{k+\frac32+\delta-s}(\mu^{-1}\nu)^{\delta_{k+2}},\quad\forall \delta\in(0,\tfrac{1}{2}),\  t\leq T_*. 
    \end{equation*} By taking $\delta<b(\delta_{k+2}-\delta_{k+3})$ small enough, we have
    \begin{equation*}
        \begin{aligned}
            \|w(\cdot,t)\|_{H^{k+\frac{3}{2}+\delta}}&\lesssim_{k,\delta} \|w(\cdot,t)\|_{L^2}+\|w(\cdot,t)\|_{\dot H^{k+\frac{3}{2}+\delta}}\\
            &\lesssim_{\ep,k,\delta} \mu^{k+\frac{3}{2}+\delta-s}(\mu^{-1}\nu)^{\delta_{k+2}}\lesssim\mu^{k+\frac{3}{2}-s}(\mu^{-1}\nu)^{\delta_{k+3}}.
        \end{aligned}
    \end{equation*}
    Together with the Sobolev embedding $H^{k+\frac{3}{2}+\delta}\hookrightarrow \dot W^{k,\infty}$ this gives \eqref{eq4.125}.
\end{proof}
\begin{rmk}
    We claim that the remainder term $w$ is also axisymmetric. 
    Since $u_0$ is axisymmetric and  $u\in C([0,T_*];H^{1})\hookrightarrow C([0,T_*];L^6)$ which is a uniqueness class for the 3D Navier-Stokes equations \cite[Theorem (3.3)]{FJR1972}, we see that $u$ must be invariant with respect to rotations about the z-axis. Since the approximate solution $\bar u$ is also axisymmetric, so is $w$.

\end{rmk}

Now let us prove a better estimate for the vector swirl of the remainder.
Denote
\begin{equation}\label{eq4.111}
\begin{aligned}
    &W_\theta:=w_\theta e_\theta,\quad \bar U_\theta:=\bar u_\theta e_\theta,\quad U_{\neq\theta}:=u_re_r+u_ze_z,\quad \bar U_{\neq\theta}:=\bar u_re_r+\bar u_ze_z,\\
    &  W_{\neq\theta}:=w_re_r+w_ze_z,\quad \mathcal{E}_\theta:=-E_\theta e_\theta,
\end{aligned}
\end{equation}
then $W_\theta$ satisfies
\begin{equation}\label{eq4.70}
    \left\{\begin{array}{l}
         \partial_tW_\theta-\Delta W_\theta +U_{\neq\theta}\cdot\nabla W_\theta+W_\theta\cdot\nabla U_{\neq\theta}=\mathcal{E}_\theta-\bar U_\theta\cdot\nabla W_{\neq\theta}-W_{\neq\theta}\cdot\nabla \bar U_\theta:=\tilde{\mathcal{E}}_\theta,\\
         W_\theta|_{t=0}=0,
    \end{array}\right.
\end{equation}
 where $\tilde{\mathcal E}_\theta$ is supported away from $\{r=0\}$. By \eqref{matrixnorm}, $W_{\neq\theta}$ also satisfies the bound in \eqref{eq4.67}
\begin{equation}\label{eq4.71}
    \|W_{\neq\theta}(\cdot,t)\|_{\dot{H}^k}\lesssim \|w(\cdot,t)\|_{\dot{H}^k}\lesssim_{\ep,k}\mu^{k-s}(\mu^{-1}\nu)^{\delta_k}\leq\mu^{k-s}(\mu^{-1}\nu)^{0.9},\quad\forall t\leq T_*.
\end{equation}
 Here we use $0<\mu^{-1}\nu<1$ and the fact that the sequence $\delta_k$
is decreasing.

To obtain the higher order energy estimate, we take $\nabla^\alpha$ on both sides of \eqref{eq4.70} for $|\alpha|=k$, and then take the inner product in $L^2$ with $\nabla^\alpha W_{\theta}$. Since $\operatorname{div}(U_{\neq\theta})=\operatorname{div}(W_\theta)=0$, the term $U_{\neq\theta}\cdot \nabla \nabla^\alpha W_{\theta}$ vanishes in the energy estimate. From the product rule, summation over $|\alpha|=k$ and H\"{o}lder's inequality, we get
\begin{equation}\label{eq4.136}
    \begin{aligned}
        &\frac{d}{dt}\| W_\theta(\cdot,t)\|_{\dot{H}^k}^2+\| W_\theta(\cdot,t)\|_{\dot{H}^{k+1}}^2\lesssim_k \|U_{\neq\theta}(\cdot,t)\|_{\dot{W}^{1,\infty}}\|W_{\theta}(\cdot,t)\|_{\dot H^k}^2\\
    &+\|W_{\theta}(\cdot,t)\|_{\dot H^k}\sum_{2\leq j, j+m=k+1}\| W_\theta(\cdot,t)\|_{\dot{H}^m}\|U_{\neq\theta}(\cdot,t)\|_{\dot{W}^{j,\infty}}+\| W_{\theta}(\cdot,t)\|_{\dot{H}^k}\| \tilde{\mathcal{E}}_{\theta}(\cdot,t)\|_{\dot{H}^k}.
    \end{aligned}
\end{equation}
Note that \eqref{eq4.15}, \eqref{eq4.125} and \eqref{matrixnorm} imply 
\begin{equation}\label{eq4.77}
    \|U_{\neq\theta}\|_{L^\infty(0,T_*;\dot{W}^{k,\infty})}\lesssim_{\ep,k}\mu^{k+1-s}\nu^{\frac{1}{2}},\quad k\in\mathbb{N}.
\end{equation}
By \eqref{eq4.77}, the first two terms on the right hand side of \eqref{eq4.136} can be bounded by
\begin{equation}\label{leadingorderest}
    C(\ep,k)\left(\mu^{2-s}\nu^{\frac{1}{2}}\|W_{\theta}(\cdot,t)\|_{\dot H^k}^2+\|W_{\theta}(\cdot,t)\|_{\dot H^k}\left(\sum_{j<k}\mu^{k-j+2-s}\nu^{\frac{1}{2}}\|W_{\theta}(\cdot,t)\|_{\dot H^j}\right)\right),\quad\forall t\leq T_*,
\end{equation}
where the constants $C(\ep,k)$ comes from the constant in \eqref{eq4.71} and \eqref{eq4.77}.
Furthermore, by the product rule, \eqref{eq4.15} and
\eqref{eq4.71} imply
\begin{equation}\label{eq4.31}
\begin{aligned}
&\|W_{\neq\theta}\cdot\nabla\bar U_\theta\|_{L^\infty(0,T_*;\dot{H}^k)}
+\|\bar U_\theta\cdot\nabla W_{\neq\theta}\|_{L^\infty(0,T_*;\dot{H}^k)}\\
&\lesssim_k\sum_{j+m=k+1}
\|W_{\neq\theta}\|_{L^\infty(0,T_*;\dot{H}^j)}\|\bar U_\theta\|_{L^\infty(0,T_*;\dot{W}^{m,\infty})}\lesssim_{k,\ep}\mu^{k+\frac32-\eta-s}\nu^{\frac12}
(\mu^{-1}\nu)^{0.9}.
\end{aligned}
\end{equation}
On the other hand, \eqref{eq4.16} and \eqref{eq4.31} give
\begin{equation*}\label{eq4.78}
\|\tilde{\mathcal E}_\theta(\cdot,t)\|_{\dot H^k}
\lesssim_{k,\ep}\mu^{k+\frac32-\eta-s}\nu^{\frac12}
(\mu^{-1}\nu)^{0.9},\qquad \forall t\leq T_*,
\end{equation*}
Substituting this and \eqref{leadingorderest} into \eqref{eq4.136} then yields

\begin{equation}\label{eq4.130}
  \begin{aligned}
      &\frac{d}{dt}\| W_{\theta}(\cdot,t)\|_{\dot H^k} ^2+\|W_\theta(\cdot,t)\|_{\dot H^{k+1}}^2\lesssim_{\ep,k} \mu^{2-s}\nu^{\frac{1}{2}}\|W_{\theta}(\cdot,t)\|_{\dot H^k}^2\\
      &\quad+\|W_\theta(\cdot,t)\|_{\dot H^k}\sum_{j< k}  \mu^{k-j+2-s}\nu^{\frac{1}{2}}\|W_{\theta}(\cdot,t)\|_{\dot H^j}+\mu^{k+\frac32-\eta-s}\nu^{\frac12}
(\mu^{-1}\nu)^{0.9}\|W_{\theta}(\cdot,t)\|_{\dot H^k},\quad\forall t\leq T_*.
  \end{aligned}
\end{equation}
In particular, if $k=0$, then the first term in the second line vanishes since there is no $j<k$.

We will now prove the following bootstrapped quantitative regularity for $W_\theta$.
\begin{lem}\label{lem4.4}
    Let $W_\theta$ to be the solution to \eqref{eq4.70} on $[0,T_*]$, then
    \begin{equation}\label{eq4.24}
        \| W_\theta(\cdot,t)\|_{\dot{H}^k}\lesssim_{\ep,k}\mu^{k-\frac{1}{2}-\eta}(\mu^{-1}\nu)^{0.9},\quad\forall t\leq T_*,\quad \forall k\in\mathbb{N}.
    \end{equation}
\end{lem}
\begin{proof}

   First we will exhibit an $L^2$ energy estimate. By taking $k=0$ in \eqref{eq4.130}, we have
    \begin{equation}\label{eq4.32}
        \partial_t\|W_\theta(\cdot,t)\|_{L^2}^2\lesssim_\ep \mu^{2-s}\nu^{\frac{1}{2}}\|W_\theta(\cdot,t)\|_{L^2}^2+\mu^{\frac{3}{2}-s-\eta}\nu^{\frac{1}{2}}(\mu^{-1}\nu)^{0.9}\|W_\theta(\cdot,t)\|_{L^2},\quad\forall t\leq T_*.
    \end{equation}
    By our choice of $T_*$ in \eqref{eq4.1} and Gronwall inequality (see Lemma 3.3 of \cite{BCD2011} for example), \eqref{eq4.32} implies
    \begin{equation}\label{eq4.33}
        \|W_\theta\|_{L^\infty (0,T_*;L^2)}\lesssim_{\ep} \mu^{-\frac{1}{2}-\eta}(\mu^{-1}\nu)^{0.9}.
    \end{equation}

   Now we are in position to derive the $\dot H^k$ estimate \eqref{eq4.24} by induction.
   The case $k=0$ has been proved in \eqref{eq4.33}. Assume \eqref{eq4.24} holds for any $k\leq K$, then \eqref{eq4.130}  implies 
   \begin{equation*}\label{eqinductivesobolevenergy}
       \begin{aligned}
           \frac{d}{dt}\|W_\theta(\cdot,t)\|_{\dot H^{K+1}}^2\lesssim_{\ep,K}\mu^{2-s}\nu^{\frac{1}{2}}\|W_{\theta}(\cdot,t)\|_{\dot H^{K+1}}^2+\|W_{\theta}(\cdot,t)\|_{\dot H^{K+1}}\mu^{K+\frac{5}{2}-\eta-s}\nu^{\frac{1}{2}}(\mu^{-1}\nu)^{0.9},
       \end{aligned}
   \end{equation*}
   and Gronwall's inequality gives
   \begin{equation*}\label{inductivefinalbound}
       \|W_\theta\|_{L^\infty (0,T_*;\dot H^{K+1})}\lesssim_{\ep,K} \mu^{K+\frac{1}{2}-\eta}(\mu^{-1}\nu)^{0.9}=\mu^{K+1-\frac{1}{2}-\eta}(\mu^{-1}\nu)^{0.9}.
   \end{equation*}
   This completes the proof of \eqref{eq4.24} by induction. \end{proof}
\subsection{Completion of the proof of Theorem \ref{thm7}}
\begin{proof}[Proof of Theorem \ref{thm7}]
 Note that for any axisymmetric function $f(r,z)$, if we denote $F(r,\theta,z)=f(r,z)e_\theta$, then there holds
\begin{equation*}
    \|f\|_{L^2}=\|F\|_{L^2},\quad \|\nabla F\|_{L^2}^2=\|\nabla f\|_{L^2}^2+\|\frac{f}{r}\|_{L^2}^2,
\end{equation*}
so Lemma \ref{lem4.4} also implies that
\begin{equation*}
    \|w_\theta(\cdot,t)\|_{L^2}\lesssim_{\ep}\mu^{-\frac{1}{2}-\eta}(\mu^{-1}\nu)^{0.9},\quad \| w_\theta(\cdot,t)\|_{\dot{H}^1}\lesssim_{\ep}\mu^{\frac{1}{2}-\eta}(\mu^{-1}\nu)^{0.9},\quad\forall t\leq T_*.
\end{equation*}
    So Lemma \ref{lem4.4} and Sobolev interpolation \eqref{eq4.122} imply
    \begin{equation*}
        \|w_\theta(\cdot,t)\|_{\dot H^{\frac{1}{2}+\eta}}\leq \|w_{\theta}(\cdot,t)\|_{L^2}^{\frac{1}{2}-\eta}\|w_{\theta}(\cdot,t)\|_{\dot{H}^1}^{\frac{1}{2}+\eta}\lesssim_{\ep} (\mu^{-1}\nu)^{0.9}.
    \end{equation*}
    Combining this with the lower bound of $\bar u_\theta$ in \eqref{eq4.34}, if we take $\mu$ large enough with respect to $\ep$, there holds
    \begin{equation*}
        \|u_\theta (\cdot,T_*)\|_{\dot H^{\frac{1}{2}+\eta}}\geq \|\bar u_\theta(\cdot,T_*)\|_{\dot H^{\frac{1}{2}+\eta}}-\|w_\theta(\cdot,T_*)\|_{\dot H^{\frac{1}{2}+\eta}}\geq C_2 \ep^{-2},
    \end{equation*}
  for some universal constant $C_2$. This together with the smallness condition of initial data \eqref{eq4.7} completes the proof of Theorem \ref{thm7} by taking $\ep\leq \min\{C_2,C_1^{-1}\}$ small enough.
\end{proof}

\section{Euler case}\label{sec4.2}
\subsection{Euler setting}
In the Euler case, we also use the coordinates and the $(u_{0r}, u_{0z})$ profile introduced in \cite{Luo24}. For this case, as in \cite{Luo24} fix $s\in(0,\frac{5}{2})$ and $b=\frac{\frac{5}{2}-s}{100}\in(0,\frac{1}{40})$. For $\mu\gg1$, we still use the parameters from \cite{Luo24}
\begin{equation}\label{eq4.43}
    \nu=\mu^{1-b}\ll \mu,\quad T_*=\ep^{-2-N}\mu^{s-2}\nu^{-\frac{1}{2}},\quad N=\max\{\frac{10}{s},100\}
\end{equation}
as in \eqref{eq4.1}. We also use the shifted $(r,z)$ space variables $(\rho,\varphi) $ as in \eqref{eq4.2}. For this case, we will choose the initial data with a different scaling for the swirl $u_{0\theta}$. Namely,  $u_0=u_{0r}e_r+u_{0\theta}e_\theta+u_{0z}e_z+u_ce_z$ with 
\begin{equation}\label{eq4.44}
    \begin{aligned}
    &u_{0\theta}=\ep^2\mu^{-2}\nu^{\frac{1}{2}}g(\mu\rho)\sin\varphi,\\
        &u_{0r}=-\ep^2\mu^{1-s}\nu^{\frac{1}{2}}f'(\mu\rho)\partial_z\rho,\\
        &u_{0z}=\ep^2\mu^{1-s}\nu^{\frac{1}{2}}f'(\mu\rho)\partial_r\rho,\\
        &u_c=\ep^2\mu^{-s}\nu^{\frac{1}{2}}\frac{f(\mu\rho)}{r},
    \end{aligned}
\end{equation}
As in Proposition \ref{prop4.1} (taking $\eta=\frac52$), the bounds
$|\nabla^j e_\theta|\lesssim_j\nu^j$ on the support of the initial data and
$\nu\ll\mu$ give
\begin{equation}\label{eq4.113}
\|u_{0\theta}\|_{\dot H^{3}}+\|u_{0\theta}e_\theta\|_{\dot H^{3}}\leq C_3\ep^2,
\end{equation}
for some universal constant $C_3$. Based on \cite{Luo24} we use the same approximate solution $\bar u$ for the Navier-Stokes case with $\eta=\frac{5}{2}$. Namely, we consider $\bar{u}$ with $\bar u_r=u_{0r}$, $\bar u_z=u_{0z}+u_c$, and with the swirl component $\bar u_\theta$ defined by
\begin{equation*}\label{eq4.45}
    \bar u_\theta=\ep^2\mu^{-2}\nu^{\frac{1}{2}}g(\mu\rho)\sin(\varphi-\ep^2t\rho^{-1}\mu^{1-s}\nu^{\frac{1}{2}}).
\end{equation*}
The estimates for $\bar u_\theta$ are then the same as in Proposition \ref{prop4.8} with $\eta=\tfrac{5}{2}.$ In particular, 
\begin{equation}\label{Eulerswirlapproximate}
\|\bar u_\theta\|_{L^\infty(0,T_*;\dot W^{k,p})}+\|\bar u_\theta e_{\theta}\|_{L^\infty(0,T_*;\dot W^{k,p})}\lesssim_{k,p} \ep^{2-Nk} \mu^{k-3}(\mu^{1-\frac{2}{p}}\nu^{\frac{1}{2}-\frac{1}{p}}).
\end{equation}
Moreover, from Remark \ref{rmkvectorinflation} with $\eta=\frac{5}{2}$ we have
\begin{equation*}\label{eq4.107}
    \|\bar{u}_\theta(\cdot,T_*)e_\theta\|_{\dot H^{1}}\geq C \varepsilon^{2-N} \mu^{-2},
\end{equation*}
for some universal constant $C$.
    
    From this, \eqref{Eulerswirlapproximate} and the interpolation of Sobolev spaces \eqref{eq4.122}, we obtain
    \begin{equation}\label{eq4.50}
    \|\bar u_\theta(\cdot,T_*)e_\theta\|_{\dot H^{3}}\geq \|\bar u_\theta(\cdot,T_*)e_\theta\|_{\dot H^1}^{3}\|\bar u_\theta(\cdot,T_*)e_\theta\|_{L^2}^{-2} \gtrsim  \ep^{-2}.
    \end{equation}
Let the force term $F=F_re_r+F_\theta e_\theta+F_ze_z$ be defined by
\begin{equation*}\label{eq4.46}
     \left\{\begin{array}{l}
             F_r=u_c\partial_zu_{0r}-\frac{1}{r}\bar u_\theta^2,  \\
             F_\theta=u_c\partial_z\bar u_\theta+\frac{1}{r}\bar u_\theta u_{0r}, \\
             F_z=u_{0r}\partial_ru_c+u_{0z}\partial_zu_c+u_c\partial_zu_c+u_c\partial_zu_{0z},
        \end{array}\right.
\end{equation*}
with respect to the new data \eqref{eq4.44}. As in \cite{Luo24} the approximate solution $\bar u$ satisfies
\begin{equation*}
    \left\{\begin{array}{l}
         \partial_t\bar u+\bar u\cdot\nabla \bar u=F+\nabla P,  \\
          \operatorname{div}\bar u=0,\\
          \bar u|_{t=0}=\bar u_0.
    \end{array}\right.
\end{equation*}
Note that $F$ is compactly supported on a ring away from the $z-$axis, and $\operatorname{dist}(\operatorname{Supp}(F),\{r=0\})\sim \nu^{-1}$.

 The estimates for $F$ and the approximate solution are then the same as in Proposition \ref{prop4.8} with $\eta=\tfrac{5}{2}.$ This is also the case for the remainder estimates for $w=u-\bar{u}$ in Proposition \ref{prop4.3} taken from \cite{Luo24}. As diffusion plays no role in the estimate in Lemma \ref{lem4.4}, we obtain the same estimate \eqref{eq4.24} for $W_{\theta}:=w_{\theta}e_{\theta}$ with $\eta=\frac{5}{2}$. Namely, 
 \begin{equation}\label{inductivefinalboundEuler}
       \|W_\theta\|_{L^\infty (0,T_*;\dot H^{k})}\lesssim_{\ep,k} \mu^{k-3}(\mu^{-1}\nu)^{0.9}.
   \end{equation}
   By the parameters chosen in \eqref{eq4.43}, we can take $\mu$ large enough depending only on $\ep$, such that \eqref{inductivefinalboundEuler} gives 
\begin{equation*}
    \|W_{\theta}(\cdot,T_*)\|_{\dot H^3}\leq C({\ep}) (\mu^{-1}\nu)^{0.9}\leq \ep.
\end{equation*}
 Since $u=\bar u+w$, and thanks to the lower bound \eqref{eq4.50}, the swirl component $u_{\theta}$ satisfies
\begin{equation*}
    \|u_\theta e_\theta(\cdot,T_*)\|_{\dot H^3}\geq \|\bar u_\theta e_\theta(\cdot,T_*)\|_{\dot H^3}-\|W_\theta(\cdot,T_*)\|_{\dot H^3}\geq C_4 \ep^{-2},
\end{equation*}
for some universal constant $C_4$. This together with \eqref{eq4.113} completes the proof of Theorem \ref{thm6} by taking $\ep\leq \min\{C_4,C_3^{-1}\}$ small enough.

\section*{Acknowledgments}

T.B. and R.H. are supported by an EPSRC New Investigator Award UKRI096 ‘Dynamics and regularity criteria for nonlinear incompressible partial differential equations'. T.B. wishes to thank Xiaoyutao Luo for a discussion about \cite{Luo24}.

\section*{Data availability statement}
We do not analyse or generate any datasets, because our work proceeds within a theoretical and mathematical approach.

\section*{AI Declaration}
The ideas in this paper are human generated mostly originating from late 2025. In August 2026,  ChatGPT 5.6 reviewed a draft, addressing typos and suggesting Cartesian variables form for \eqref{eq4.70}. The paper was written entirely by humans.

\end{document}